\documentclass{amsart}
\usepackage{a4wide}
\usepackage{amsfonts,amssymb,amsmath,mathtools,mathrsfs,bm,stmaryrd,amsthm,enumerate}
\usepackage [latin1]{inputenc}
\usepackage[hidelinks]{hyperref}
\usepackage[T1]{fontenc}
\usepackage[english]{babel}
\usepackage[capitalise]{cleveref}
\usepackage{comment}
\usepackage[
backend=biber,
style=alphabetic,
maxbibnames=99,
doi=false
]{biblatex}

\bibliography{main}

\usepackage{xurl}
\hypersetup{breaklinks=true}
\usepackage[parfill]{parskip}
\usepackage{wrapfig,tikz, tikz-cd}
\usetikzlibrary{arrows}

\RequirePackage{tikz-cd}
\RequirePackage{amssymb}

\usetikzlibrary{hobby}    
\usetikzlibrary{calc}
\tikzset{curve/.style={settings={#1},to path={(\tikztostart)
      .. controls ($(\tikztostart)!\pv{pos}!(\tikztotarget)!\pv{height}!270:(\tikztotarget)$)
      and ($(\tikztostart)!1-\pv{pos}!(\tikztotarget)!\pv{height}!270:(\tikztotarget)$)
      .. (\tikztotarget)\tikztonodes}},
  settings/.code={\tikzset{quiver/.cd,#1}
    \def\pv##1{\pgfkeysvalueof{/tikz/quiver/##1}}},
  quiver/.cd,pos/.initial=0.35,height/.initial=0}
\tikzset{tail reversed/.code={\pgfsetarrowsstart{tikzcd to}}}
\tikzset{2tail/.code={\pgfsetarrowsstart{Implies[reversed]}}}
\tikzset{2tail reversed/.code={\pgfsetarrowsstart{Implies}}}
\tikzset{no body/.style={/tikz/dash pattern=on 0 off 1mm}}

\newtheorem{theorem}{Theorem}[section]
\newtheorem{theorema}{Theorem}

\newtheorem{lemma}[theorem]{Lemma}
\newtheorem{proposition}[theorem]{Proposition}
\newtheorem{corollary}[theorem]{Corollary}

\theoremstyle{definition}
\newtheorem{definition}[theorem]{Definition}
\newtheorem{remark}[theorem]{Remark}
\newtheorem{example}[theorem]{Example}

\newcommand{\parall}{//}
\newcommand{\Hom}{{\operatorname{Hom}}}

\newcommand{\HH}{{\operatorname{HH}}}

\newcommand{\Ext}{{\operatorname{Ext}}}

\DeclareMathOperator{\Inn}{Inn}

\newcommand\sbu[1][.5]{\mathbin{\vcenter{\hbox{\scalebox{#1}{$\bullet$}}}}}

\DeclareMathOperator{\Der}{Der}

\DeclareMathOperator{\op}{op}
\DeclareMathOperator{\rad}{R}

\def\sl{\mathfrak{sl}}

\def\A{A}
\def\B{\mathcal{B}}

\title[Hochschild  cohomology]{On the first relative Hochschild  cohomology}
\author{Jonathan Lindell, Lleonard Rubio y Degrassi}
\date{}

\begin{document}

\begin{abstract}
  In this paper we investigate the Lie algebra structure of the first relative Hochschild cohomology. Let $A,B$ be finite-dimensional basic $k$-algebras over an algebraically closed field of characteristic zero, such that $Q_B$ is a subquiver of $Q_A$. We show that if the complement of $Q_A$ by the arrows of $Q_B$ is a simple directed graph, then the first relative Hochschild cohomology $\HH^1(A|B)$ is a solvable Lie algebra.  We also compute the Lie algebra structure of the first relative Hochschild cohomology for radical square zero algebras and for dual extension algebras of directed monomial algebras. 
\end{abstract}

\maketitle

\section{Introduction} 

Hochschild cohomology, together with its plethora of structures, stands as a crucial invariant under diverse categorical equivalences \cite{Keller,Keller2,HZS, Rou, KLZ, BR2}. In particular, the (restricted) Lie algebra structure of the first Hochschild cohomology is invariant under derived equivalences. Although Hochschild cohomology is a powerful invariant, even our understanding of the dimension and the Lie algebra structure of the first Hochschild cohomology of finite-dimensional algebras remains quite limited, see for example \cite{Linckelmann2}. In the past years new methods have been introduced to improve our understanding of the graded Lie algebra structure of Hochschild cohomology \cite{NVW,Sua,ER, RSS,LinRyD, CSSS}. One of the main difficulties in computing Hochschild cohomology is due to the fact that Hochschild cohomology is not functorial, so there is no direct way to study Hochschild cohomology using quotient or subalgebras. In particular, this prevents reducing the study of the Lie algebra structure of the first Hochschild cohomology to smaller, and potentially more tractable, Lie subalgebras.
However, in the case of a quotient modulo an idempotent ideal there are some instances of functoriality properties of Hochschild cohomology that emerge.  Indeed, this is the case for one-point extension studied by Happel \cite{Ha} and generalized to the case of triangular matrix algebras by Michelena and Platzeck in \cite{MichelenaPlatzeck2000}, by Green and Solberg in \cite{GreenSolberg2002} and by Cibils, Marcos, Redondo and Solotar in \cite{CibilsMarcosRedondoSolotar2003}. The work of Happel's long exact sequence has also been generalized by de la Pe\~na and Xi to the case of algebras with heredity ideals. All these works have been extended to algebras with stratifying ideals in \cite{KoenigNagase2009}. Furthermore, this was extended even further, by Han \cite{Han2014} and Hermann \cite{Hermann2016}, by studying recollements of the module and derived categories respectively. The case for subalgebras remains more elusive.

In a seminal paper \cite{Hochschild1956}, Hochschild defines relative Hochschild (co)homology.  The notion of relative Hochschild cohomology accounts for the additional structure introduced by the presence of a subalgebra.  Gerstenhaber and Schack used relative Hochschild cohomology in the context of deformation theory, see \cite{GerstenhaberSchack1986,GerstenhaberSchack1988}. Although, this cohomology theory is seldomly employed due to its limited direct connection with Hochschild cohomology, recent developments \cite{CLMS,CLSS, CibilsLanzilottaMarcosSolotar2024}, in the context of Han's conjecture and adding arrows to a quiver, have sparked a growing interest in the study of relative Hochschild (co)homology.  

Let $A$ be a finite-dimensional algebra and $B$ a subalgebra of $A$. The first goal of this paper is to study the relation between the Lie algebra structure of $\HH^1(A|B)$, the first relative Hochschild cohomology of $A$, and $\HH^1(A)$, the first Hochschild cohomology of $A$. More precisely, in Lemma \ref{lemma:emb} we provide an embedding $\HH^1(A|B)\to \HH^1(A)$ which allows us to transfer Lie theoretic properties of $\HH^1
(A)$, such as the solvability studied in \cite{RSS,LinRyD,ER}, to $\HH^1
(A|B)$. However, it is quite common for $\HH^1(A)$ to be non-solvable while $\HH^1(A|B)$ is solvable; see Example \ref{Ex:notsolv}. Our first main theorem provides a sufficient condition for the solvability of $\HH^1(A|B)$ in case where $A$ and $B$ are basic algebras:

\begin{theorema}[Theorem \ref{thm1}]\label{mainthm}
  Let $k$ be an algebraically closed field of characteristic zero. Let $A$ be a finite-dimensional basic algebra. Let $B$ be a basic algebra which is a subalgebra of $A$ such that the set of arrows of $Q_B$ is a subset of the arrows of $Q_A$ and such that $A$ and $B$ have the same semisimple part. If the complement of the quiver of $A$ by the arrows of the quiver of $B$ has no parallel arrows and at most one loop at each vertex, then $\HH^1(A|B)$ is a solvable Lie algebra. 
\end{theorema}

We also establish an analogue in positive characteristic, but then we have to require that the quiver of $A$ does not have any loops, see Theorem \ref{thm2}.

One essential ingredient for the proof of Theorem \ref{mainthm} is the description of the Lie algebra structure of the first relative Hochschild cohomology for radical square zero algebras. For this reason, we first focus on combinatorial methods to compute the first relative Hochschild cohomology for monomial algebras using Strametz's work \cite{Str}, see Lemma \ref{relcomb}, and then we apply these tools to radical square zero basic algebras, see Theorem \ref{radsqrzero}.

The next goal of this work is to obtain a decomposition of $\HH^1(A)$ in terms of $\HH^1(B)$ and $\HH^1(A|B)$ for some family of algebras. As a consequence, we show how to recover the solvability of a certain Lie subalgebra of $\HH^1(A)$ from both $\HH^1(B)$ and $\HH^1(A|B)$. For this reason we focus on a family of quasi-hereditary algebras.

Quasi-hereditary algebras were first defined by Scott in \cite{Scott1987}, motivated by representation theory of algebraic groups, whose module categories are standard examples of highest weight categories, introduced in \cite{ClineParshallScott1988}. Classical examples of quasi-hereditary algebras are hereditary algebras, algebras of global dimension two, blocks of category $\mathcal{O}$ and Schur algebras. Koenig defined exact Borel subalgebras of quasi-hereditary algebras in \cite{Koenig1995I}. It was shown by Koenig, K\"{u}lshammer and Ovsienko that every quasi-hereditary algebra is Morita equivalent to a quasi-hereditary algebra with a regular exact Borel subalgebra \cite{KKO}. In addition, Rodriguez Rasmussen recently proved that the regular exact Borel subalgebra is unique up to an inner automorphism of the quasi-hereditary algebra, see \cite{Rasmussen}. One example of combinatorically well-behaved quasi-hereditary algebras with exact Borel subalgebras are the dual extension algebras, first defined by Xi in \cite{Xi1994}.  The dimension and the algebra structure of Hochschild cohomology of dual extensions has been studied in \cite[Example 4]{dPX} and in \cite[\S 2.4]{XZ}, however not its Lie structure. We study the first Hochschild cohomology of dual extensions of monomial directed algebras. Given two directed monomial algebras $A$ and $B$ we obtain an exact sequence of Lie algebras 
\begin{equation*}
  \begin{tikzcd}
    0 \ar{r} & \mathcal{J}/\mathcal{I} \ar{r} & \HH^1(\Lambda) \ar{r} & \HH^1(B) \ar{r} & 0 
  \end{tikzcd}
\end{equation*}
where $\Lambda = \Lambda(B,A^{\op})$ is the dual extension of $A$ and $B$, in which $B$ is the exact Borel subalgebra and $\mathcal{J}/\mathcal{I}$ is a Lie algebra related to $\HH^1(\Lambda|B)$, see Theorem \ref{Dualmain} for further details.  In general $\mathcal{J}/\mathcal{I}$ is not isomorphic to $\HH^1(\Lambda|B)$, see Example \ref{ex:notcong}. However, if for example the quiver of $B$ is of type $A_n$ then $\mathcal{J}/\mathcal{I}\cong \HH^1(\Lambda|B)$, where $\Lambda=\Lambda(B,B^{\op})$.  For further details see Example \ref{ex:string}.

For a monomial algebra $\Lambda$ and a monomial subalgebra $B$ we have that $\HH^1(\Lambda)$ and $\HH^1(\Lambda|B)$ are graded Lie algebras, where the grading is induced from the lengths of the paths. We denote the degree one component of $\HH^1(\Lambda)$ and $\HH^1(\Lambda|B)$ by $\HH^1(\Lambda)_{|_1}$ and $\HH^1(\Lambda|B)_{|_1}$, respectively. In degree 1 we obtain from the above short exact sequence the following result:

\begin{theorema}[Theorem \ref{thm:deg-one}]\label{mainthm1.1}
  Let $A$ and  $B$ be directed algebras and let $\Lambda=\Lambda(B,A^{\op})$. Then we have the following isomorphism as Lie algebras
  \[
    \HH^1(\Lambda)_{|_1}\cong \HH^1(B)_{|_1}\oplus \HH^1(\Lambda|B)_{|_1}.
  \]
\end{theorema}

\subsection*{Outline.} In Section \ref{background} we recall some basic background on relative homological algebra and (relative) Hochschild cohomology. In Section \ref{relhochmon} we provide a combinatorial description of the first relative Hochschild cohomology for monomial algebras using Strametz's methods \cite{Str}. In Section \ref{sec:embandsol} we prove Theorem \ref{mainthm}.  In Section \ref{dualext} we provide some applications for dual extension algebras and we prove Theorem \ref{mainthm1.1}.

\subsection*{Acknowledgements.} We would like to thank Julian K\"ulshammer for his support and for making some important suggestions. We would also like to thank Steffen Koenig for several important comments. Finally, we thank the referee for the careful reading of the paper and for the helpful suggestions. The second author participates in the INdAM group GNSAGA.

\section{Background}\label{background}

Let $k$ be an algebraically closed field and let $\A$ be a finite-dimensional $k$-algebra. Let $\A^e=\A\otimes \A^{\op}$ be the enveloping algebra of $\A$. 

\subsection*{Relative homological algebra} We recall the definitions within relative homological algebra as defined by Hochschild, see \cite{Hochschild1956}. Let $\A$ be a $k$-algebra and let $B \subseteq \A$ be a $k$-subalgebra. A \textit{short} $(\A | B)$\textit{-exact sequence} is a short exact sequence of $\A$-modules that splits over $B$
\begin{equation*}
  \begin{tikzcd}
    0 \ar{r} & X \ar{r}{f} & Y \ar{r}{g} \ar[dashed, bend left]{l} & Z \ar{r} \ar[dashed, bend left]{l} & 0.
  \end{tikzcd}
\end{equation*}
When the choice of algebra and subalgebra is clear, we will refer to short $(\A | B)$-exact sequences as being \textit{relative short exact sequences}. Similarly, we say that an epimorphism is \textit{relative split}, if it splits over the subalgebra. We say that a complex of $\A$-modules is \textit{relative exact} if the complex is exact and the short exact sequences which make up the complex are relative exact, i.e. the short exact sequences of the form 
\begin{equation*}
  \begin{tikzcd}
    0 \ar{r}{} & \mathrm{Ker}(d_n) \ar[hook]{r}{} & M_n \ar{r}{d_n} & \mathrm{Ker}(d_{n+1}) \ar{r} & 0 
  \end{tikzcd}
\end{equation*}
are relative short exact.

An $\A$-module $P$ is called \textit{relative projective} if for every relative split epimorphism $g : Y \rightarrow Z$ and every $\A$-module map $\varphi : P \rightarrow Z$ there exists an $\A$-module map $\psi : P \rightarrow Y$ such that $\varphi = g \circ \psi$. 

\begin{proposition}
  Let $M$ be a $B$-module. Then $A \otimes_B M$ is a relative projective module. Conversely, if 
  $P$ is a relative projective module, then there exists a $B$-module $M$ such that $P$ is a summand of $A \otimes_B M$. 
\end{proposition}

For a proof, see Lemma 2 in \cite{Hochschild1956}.

A \textit{relative projective resolution} of an $\A$-module $M$ is a relative exact complex 
\begin{equation*}
  \begin{tikzcd}
    ... \ar{r} & P_2 \ar{r}{d_2} & P_1 \ar{r}{d_1} & P_0 \ar{r}{\varepsilon} & M \ar{r} & 0 
  \end{tikzcd}
\end{equation*}
such that $P_i$ is relative projective for all $i \geq 0$. Every module has a relative projective resolution by the following process. Let $M$ be an $A$-module, then we have a relative split epimorphism
$$
\varphi : A \otimes_B M \rightarrow M
$$
where the relative splitting is given by the $B$-module map $m \mapsto 1 \otimes_B m$. Let $K = \mathrm{Ker}(\varphi)$, then similarly we have a relative split epimorphism onto $K$. Thus, by pasting all these together, we get a relative exact resolution of $M$. Akin to the absolute case, we define the \textit{relative Ext functor} to be $\Ext_{\A|B}^n(M,N) = H^{n}(\Hom_{\A}(P_{\sbu}, N))$, where $P_{\sbu}$ is a relative projective resolution of $M$. This is independent of the choice of relative projective resolution, see Section 2 in \cite{Hochschild1956}. 

\subsection*{Relative Hochschild cohomology} The $n$-th Hochschild cohomology of $\A$, denoted by $\HH^n(\A)$, is defined as $\mathrm{Ext}^n_{\A^e}(\A,\A)$. Let $B \subseteq \A$ be a $k$-subalgebra, and let $X$ be an $\A$-bimodule.  The $n$-th relative Hochschild cohomology is defined as
\[ 
  \HH^n(\A|B,X) = \mathrm{Ext}^n_{\A^e| B^e}(\A,X)
\]
for $n \geq 0$. In order to interpret and handle these groups similarly to the non-relative setting, one uses the relative bar resolution, see Section 3 in \cite{Hochschild1956}.

\begin{proposition}{}
  The relative Hochschild cohomology $\HH^*(\A|B,M)$ is the cohomology
  of the complex of cochains 

  \begin{tikzcd}
    0 \arrow[r] &  \mathrm{Hom}_{B^e}(B,M) \arrow[r, "\delta^{0}"] & \mathrm{Hom}_{B^e}(\A,M) \arrow[r, "\delta^{1}"] & \mathrm{Hom}_{B^e}(\A \otimes_B \A,M) \arrow[r, "\delta^{2}"] & \cdots
  \end{tikzcd}	

  where for $n \geq 1$
  \begin{equation*}
    \begin{split}
      \delta(f)(x_1\otimes \dots \otimes x_n)&=x_1f(x_2\otimes \dots \otimes x_n)+
      \sum_{i=1}^{n-1} (-1)^i f(x_1\otimes \dots \otimes x_ix_{i+1}\otimes \dots \otimes x_n)\\
      +&(-1)^{n}f(x_1\otimes \dots \otimes x_{n-1})x_n 
    \end{split}
  \end{equation*}
  for $n=0$ we have: $\delta(f)(x)=xf(1)-f(1)x$.
\end{proposition}

\begin{remark}\label{relativeHH1}
  By the above proposition one has that $\HH^0(\A |B) \cong \text{Z}(\A)$ and one can view $\HH^1(\A|B,\A)$ as the set of $k$-derivations which are $B^e$-module homomorphisms quotient by the set of inner derivations which are $B^e$-modules homomorphisms. In other words, 
  \[\HH^1(\A|B,\A)\cong \mathrm{Der}_{B^e}(\A)/\mathrm{Inn}_{B^e}(\A) 
  \]
  where $\mathrm{Der}_{B^e}(\A):=\mathrm{Der}_k(\A)\cap \Hom_{B^e}(\A, \A)$ and $\mathrm{Inn}_{B^e}(\A):=\mathrm{Inn}_k(\A)\cap \Hom_{B^e}(\A, \A)$. Any relative inner derivation must therefore be given by an element that commutes with all elements in $B$. 
\end{remark}

\begin{lemma} {\rm\cite[Lemma 3.2]{CLSS}}\label{restB}
  Let $f:\A\to X$ be a derivation. The map $f$ is a $B^e$-module
  morphism if and only if $f|_B=0$
\end{lemma}

\begin{lemma}
  \label{lem:subideal}
  The space $\Der_{B^e}(A)$ is a Lie subalgebra of $\Der(A)$ and $\Inn_{B^e}(A)$ is a Lie ideal of $\Der_{B^e}(A)$.
\end{lemma}

\begin{proof}
  The statement follows from Lemma \ref{restB}.
\end{proof}

\begin{remark}
  Under the assumption that $A/B$ is projective as a $B$-bimodule, there is a long exact sequence 
  \begin{equation*}
    \begin{tikzcd}
      ... \ar{r} & \HH_n(B) \ar{r} & \HH_n(A) \ar{r} & \HH_n(A|B) \ar{r} & \HH_{n-1}(B) \ar{r} & ... 
    \end{tikzcd}
  \end{equation*}
  ending in the degree $1$, see \cite{Kaygun}. Unfortunately, there is no such long exact sequence for Hochschild cohomology under the same assumptions. The closest there is, is a long exact sequence 
  \begin{equation*}
    \begin{tikzcd}
      ... \ar{r} & \HH^n(A|B) \ar{r} & \HH^n(A) \ar{r} & \HH^n(B,A) \ar{r} & \HH^{n+1}(A|B) \ar{r} & ... 
    \end{tikzcd}
  \end{equation*}
  starting in degree $1$ under the same assumptions, see \cite{Kaygun}. If $A/B$ is also injective, we have that $\HH^n(B,A) \cong \HH^n(B)$ for $n \geq 2$, giving us an exact sequence relating the Hochschild cohomologies, starting in degree $2$. Recently, Cibils, Lanzilotta, Marcos and Solotar extended this to a long nearly exact sequence for Hochschild homology, weakening the assumption that $A/B$ is projective as a $B$-bimodule, see \cite{CibilsLanzilottaMarcosSolotar2024, CibilsLanzilottaMarcosSolotar2021, CibilsLanzilottaMarcosSolotar2021C}. For their construction in the case of Hochschild cohomology, see \cite{Lin}. 
\end{remark}

\subsection*{Hochschild cohomology of monomial algebras}
Let $Q$ be a finite quiver having $Q_0$ as set of vertices and $Q_1$ as set of arrows. We compose arrows in a quiver 
from right to left, that is, for arrows $\alpha$ and $\beta$ we write $\beta\alpha$ for the path from the source of $\alpha$ to the target
of $\beta$.  Let $\A$ be a finite-dimensional monomial $k$-algebra, that is, $A$ is a path algebra $kQ_{A}$ quotient an admissible ideal $I_{A}$ generated by $Z$ a set of paths of length at least two.  We always assume that no proper subpath of a path in $Z$ is again in $Z$. Let $\mathcal{B}=\mathcal{B}_{\A}$ be the set of paths of $Q$ which do not contain any element of $Z$ as a subpath. Clearly, $\mathcal{B}$ forms a basis of $\A$. We denote by $E \cong kQ_{0}$ the separable subalgebra of $\A$. For a monomial algebra $A$, one can use the minimal projective resolution of the $A$-bimodule $A$ given by Bardzell \cite{Ba} to compute Hochschild cohomology. Applying the contravariant functor $\mathrm{Hom}_{A^e}(-,A)$ to this resolution one obtains a cochain complex denoted by $C_{\min}$.

We recall now the combinatorial description of $\HH^1(\A)$ due to Strametz \cite{Str} using Bardzell's resolution.  Recall that, for a subset $K$ of paths in $Q$, 
\begin{equation*}
  k(K\parall\mathcal{B}) 
\end{equation*}
is defined to be the $k$-vector space spanned by all pairs of paths $p\parall q$ such that $p \in K$ and $q$ is a path in $\mathcal{B}$ that is parallel to $p$. For a relation $r$, we denote by $r^{a \parall  \gamma}$ the sum of all paths in $\mathcal{B}$ obtained by replacing each appearance of the arrow $a$ in $r$ by the path $\gamma$. 

\begin{proposition}\label{Stra2.6} {\rm(\cite[Proposition 2.6]{Str})}
  Let $\A$ be a finite-dimensional monomial algebra. Then the cochain complex $\mathcal{C}_{\min}$ is isomorphic to the following cochain complex $\mathcal{C}_{\operatorname{mon}}$: 
  \begin{center}
    \begin{tikzcd}
      0 \arrow[r] & k(Q_{0}\parall  \mathcal{B}) \arrow[r, "\delta^{0}"] & k(Q_{1}\parall  \mathcal{B}) \arrow[r, "\delta^{1}"] & k(Z\parall  \mathcal{B}) \arrow[r, "\delta^{2}"] & \cdots,
    \end{tikzcd}	
  \end{center}
  where the differentials in degree $0$ and $1$ are defined as follows:
  \begin{equation*}
    \begin{split}
      \delta^{0}:k(Q_{0}\parall  \mathcal{B}) &\to k(Q_{1}\parall  \mathcal{B})\\
      e\parall  p &\mapsto \sum_{s(a)=e,ap \in \mathcal{B}}a\parall  ap - \sum_{t(a)=e,p a \in \mathcal{B}}a\parall  p a,\\
      \delta^{1}:k(Q_{1}\parall  \mathcal{B}) &\to k(Z \parall  \mathcal{B})\\
      a\parall  p &\mapsto \sum_{r \in Z}r\parall  r^{a \parall  p},
    \end{split}
  \end{equation*}
  In particular, we have $\HH^{1}(\A)\cong\mathrm{Ker}(\delta^{1})/ \mathrm{Im}(\delta^{0})$
  as $k$-spaces.
\end{proposition}

With this differential, one can endow $\mathrm{Ker}(\delta^{1})/ \mathrm{Im}(\delta^{0})$ with a Lie algebra structure \cite[Theorem 2.7]{Str}. Recall that $\HH^{1}(\A)=\oplus_{i=0}^{n} \HH^{1}(\A)_{|_i}$ is a graded Lie algebra, 
where the grading is induced from the grading of $\A$. Note that our grading is shifted by $1$ compared to  \cite[\S 4]{Str}. 

\subsection*{Lie theory} We recall the basic definitions of solvable Lie algebras and semidirect products, as defined in for example \cite{Erdmann}.

\begin{definition}
  Let $\mathfrak{g}$ be a Lie algebra. We say that $\mathfrak{g}$ is \textit{solvable} if the derived series 
  \begin{equation*}
    \mathfrak{g} \geq[\mathfrak{g},\mathfrak{g}] \geq [[\mathfrak{g},\mathfrak{g}],[\mathfrak{g},\mathfrak{g}]] \geq ... 
  \end{equation*}
  eventually becomes zero. We denote the iterated brackets by $\mathfrak{g}^{(i)}$, where we define $\mathfrak{g}^{(0)} = \mathfrak{g}$. 
\end{definition}

\begin{definition}
  A Lie algebra $\mathfrak{g}$ is \textit{strongly solvable} if its derived subalgebra $[\mathfrak{g},\mathfrak{g}]$ is nilpotent.
\end{definition}

If $k$ is algebraically closed of characteristic zero, as a consequence of Lie's theorem, a Lie algebra is strongly solvable if and only if it is solvable. In positive characteristic, every strongly solvable algebra is solvable, but the converse does not hold, see for example Page 96 in \cite{Sel}. 

\begin{definition}
  The \textit{radical} of a Lie algebra is the maximal solvable ideal. For a Lie algebra $\mathfrak{g}$ we denote this by $\rad(\mathfrak{g})$.
\end{definition}

\begin{definition}
  Let $\mathfrak{g}$ be a Lie algebra. We say that $\mathfrak{g}$ is a \textit{semidirect product} of a subalgebra $\mathfrak{h}$ and ideal $\mathfrak{r}$ if it decomposes as a sum $\mathfrak{g} = \mathfrak{h} \oplus \mathfrak{r}$ of vector spaces. We then denote it by $\mathfrak{g} = \mathfrak{h} \ltimes \mathfrak{r}$. 
\end{definition}

\begin{theorem}\label{thm:Levi}
  Let $\mathfrak{g}$ be a Lie algebra over a field of characteristic zero. Then 
  \begin{equation*}
    \mathfrak{g} \cong \mathfrak{g}/\rad(\mathfrak{g}) \ltimes \rad(\mathfrak{g})
  \end{equation*}  
\end{theorem}

For a proof see, for example, Levi's Theorem on Page 91 in \cite{Jacobson1962}.

\begin{remark}
  We call $\mathfrak{g}/\rad(\mathfrak{g}) \ltimes \rad(\mathfrak{g})$ the \textit{Levi decomposition} of $\mathfrak{g}$.  
\end{remark}

\begin{example}\label{Example:deleteRows}
  Let $\mathfrak{g} \subseteq \mathfrak{gl}_n$ be the subspace of matrices where the first $m$ rows are zero. We can write every element of $\mathfrak{g}$ as a pair $(\mathcal{A},\mathcal{B})$ where $\mathcal{A}$ is the $(n-m) \times m$ matrix given by the last $n-m$ rows and first $m$ columns, and $\mathcal{B}$ is the $(n-m) \times (n-m)$ square matrix given by last $n-m$ rows and columns. Let $x = (\mathcal{A},\mathcal{B}),y = (\mathcal{A}',\mathcal{B}') \in \mathfrak{g}$, then 
  \begin{equation*}
    [x,y] = 
    \begin{pmatrix}
      0 & 0 \\
      \mathcal{B}\mathcal{A}' - \mathcal{B}'\mathcal{A} & [\mathcal{B}, \mathcal{B}']
    \end{pmatrix}.
  \end{equation*}
  Note thus that $\mathfrak{g}$ is a semidirect product of $\mathfrak{gl}_m$ and the ideal given by $(\mathcal{A},0)$. From the bracket, note that if $\mathfrak{h} \subseteq \mathfrak{gl}_m$ is an ideal then $(\mathcal{A},0) + \mathfrak{h}$ is an ideal of $\mathfrak{g}$. Furthermore, $\mathfrak{h}$ is solvable if and only if $(\mathcal{A},0) + (\mathfrak{h},0)$ is solvable. Note that $(\mathcal{A},0)$ is solvable and part of the radical. Therefore, we get that the radical of $\mathfrak{g}$ is $ (\mathcal{A},0) + \rad(\mathfrak{gl}_m)$. 
\end{example}

\section{Relative Hochschild cohomology for monomial algebras}
\label{relhochmon}
\subsection{Combinatorial description of the relative $\HH^1$} The goal of this subsection is to provide a combinatorial description of the first relative Hochschild cohomology in the case of monomial algebras, as in Strametz's work \cite{Str}.  We briefly recall the notation: Let $\A$ be a finite-dimensional monomial $k$-algebra.  We denote by $E \cong kQ_{0}$ the separable subalgebra of $\A$. Throughout this section we assume that $A$ is connected. We compute $\HH^1(A)=\mathrm{Ker}(\delta^1_A)/\mathrm{Im}(\delta^0_A)$ following the description of Proposition \ref{Stra2.6}.

\begin{remark}\label{relker}
  By \cite[Proposition 2.8]{Str} we have that the Lie algebras 
  $\mathrm{Ker}(\delta^1_A)$ and $\mathrm{Der}_{E^e}(\A)=\mathrm{Der}(\A)\cap \mathrm{Hom}_{E^e}(\A,\A)$ (with the bracket defined as $f\circ g-g\circ f$) are isomorphic. More precisely, take $f\in \mathrm{Der}_{E^e}(\A)$. Assume that $f(a)=\sum_{a\parall p\in (Q_{\A})_1\parall \B} \lambda_{a\parall p}p$, where $a \in (Q_{\A})_1$, $\lambda_{a\parall p}\in k$. Then 
  \[
    f \mapsto \sum_{a\parall p\in (Q_{\A})_1\parall \B}\lambda_{a\parall p}a\parall p
  \]
  defines an isomorphism. The isomorphism restricts to $\mathrm{Im}(\delta^0_A)$ and $\mathrm{Inn}_{E^e}(\A)$, see \cite[Corollary 2.9]{Str} and \cite{GerstenhaberSchack1992}.
\end{remark}

\begin{definition}
  Let $B$ be a monomial subalgebra of $A$ such that $Q_B$ is a subquiver of $Q_A$. We denote by $\mathrm{Ker}(\delta^1_{\A|B})$ the set of elements $f$ in $\mathrm{Ker}(\delta^1_{\A})$ such that $f$ restricted to the arrows of $Q_B$ is zero. Similarly, we denote by $\mathrm{Im}(\delta^0_{\A|B})$ the set of elements $g$ in $\mathrm{Im}(\delta^0_{\A})$ such that $g$ restricted to the arrows of $Q_B$ is zero. In other words, $\mathrm{Im}(\delta^0_{\A|B})= \mathrm{Im}(\delta^0_{\A}) \cap \mathrm{Ker}(\delta^1_{\A|B})$.
\end{definition}

\begin{remark}\label{rem:relcomb}
  Let $B$ be a monomial subalgebra of $A$ such that $Q_B$ is a subquiver of $Q_A$. Let $f:=\sum_{a\parall \gamma\in (Q_{\A})_1\parall \mathcal{B}}\lambda_{a,\gamma} a\parall \gamma$ be an element in $\mathrm{Ker}(\delta^1_{\A})$. If $f\in \mathrm{Ker}(\delta^1_{\A|B})$, then
  \[
    f=\sum_{a\parall \gamma\in (Q_{\A})_1\parall \mathcal{B},\ a\notin (Q_B)_1}\lambda_{a,\gamma} a\parall \gamma.
  \]
  Elements in $\mathrm{Im}(\delta^0_{\A|B})$ can be described similarly. Clearly, $\mathrm{Im}(\delta^0_{\A|B})\subseteq \mathrm{Ker}(\delta^1_{\A|B})$. 
\end{remark}

We describe $\HH^1(\A|B)$ in terms of the complex defined by Strametz. 

\begin{lemma}\label{relcomb}
  Let $\A$ be a monomial algebra.  Let $B$ be a monomial subalgebra of $A$ such that $Q_B$ is a subquiver of $Q_A$. Then $\mathrm{Ker}(\delta^1_{\A|B})$ is a Lie subalgebra of $\mathrm{Ker}(\delta^1_{\A})$ and $\mathrm{Im}(\delta^0_{\A|B})$ is a Lie ideal of $\mathrm{Ker}(\delta^1_{\A|B})$. In addition, $\HH^1(\A|B)$ is isomorphic as a Lie algebra to $\mathrm{Ker}(\delta^1_{\A|B})/\mathrm{Im}(\delta^0_{\A|B})$. 
\end{lemma}

\begin{proof}
  By Remark \ref{relativeHH1}, we have $\HH^1(\A|B,\A)\cong \mathrm{Der}_{B^e}(\A)/\mathrm{Inn}_{B^e}(\A)$. Since $B$ contains $E$, then $\mathrm{Der}_{B^e}(\A)=\mathrm{Der}_k(\A)\cap \Hom_{B^e}(\A, \A)=\mathrm{Der}_{E^e}(\A)\cap \Hom_{B^e}(\A, \A)$. By Lemma \ref{restB}, the space $\mathrm{Der}_{B^e}(\A)$ consists of $k$-derivations which are $E^e$-module homomorphisms such that when restricted to $B$ are zero. Then, by the isomorphism $\mathrm{Der}_{E^e}(\A)\cong \mathrm{Ker}(\delta^1_{\A})$, the elements in $\mathrm{Der}_{B^e}(\A)$ correspond to elements in $\mathrm{Ker}(\delta^1_{\A})$ which vanish when restricted to the arrows of $B$, that is, to $\mathrm{Ker}(\delta^1_{\A|B})$. The same holds for $\mathrm{Im}(\delta^0_{\A|B})$. The statement follows.
\end{proof}

\subsection{Relative Hochschild cohomology for radical square zero algebras} The goal of this subsection is to study the Lie algebra structure of $\HH^1(A|B)$ where $A, B$ are radical square zero algebras such that $Q_B$ is a subquiver of $Q_A$. 

Recall that the \textit{$n$-Kronecker quiver} is a quiver consisting of two vertices, say $e$ and $f$, and $n$ arrows from $e$ to $f$. The $n$-\textit{bouquet quiver} is a quiver consisting of a single vertex and $n$ loops. The $n$-\textit{bouquet} is the radical square zero algebra associated with the $n$-bouquet quiver.
\begin{lemma}
  \label{lem:nkron}
  Let $m$ be a non-negative integer and $n$ a positive integer such that $m\leq n$. Let $A,B$ be the $n$ and the $m$-Kronecker quivers (or the $n$ and the $m$-bouquets,
  excluding the case when the characteristic of $k$ is $2$ and $Q_A$ is a quiver with one loop $\alpha$ and $m=0$), respectively. Then $\mathrm{Ker}(\delta^1_{A|B})\cong \mathfrak{I}\rtimes \mathfrak{gl}_{n-m}(k)$, where $\mathfrak{I}$ is an abelian Lie ideal of $\mathrm{Ker}(\delta^1_{A|B})$ isomorphic to $k^{m(n-m)}$. 
\end{lemma}

\begin{proof}
  Assume $n>1$. A basis for $\mathrm{Ker}(\delta^1_{A|B})$ is given by 
  $$\{ \alpha_i\parall \alpha_j\  |\ m+1\leq i\leq n,\  1\leq j\leq n\}.$$ 
  This basis can be seen as the basis obtained by deleting the first $m$ rows of $\mathfrak{gl}_n(k)$.  Note that if $m=0$, then $\mathrm{Ker}(\delta^1_{A|B})=\mathrm{Ker}(\delta^1_A)$ which is isomorphic to $\mathfrak{gl}_n(k)$. The Lie subalgebra of $\mathrm{Ker}(\delta^1_{A|B})$ generated by $$\{ \alpha_i\parall \alpha_j\  |\ m+1\leq i, j\leq n\}$$ is isomorphic to $\mathfrak{gl}_{n-m}(k)$. We denote by $\mathfrak{I}$ the subspace generated by $$\{ \alpha_i\parall \alpha_j\  |\ m+1\leq i\leq n,\  1\leq j\leq m\}$$ and we claim that this is an abelian Lie ideal of $\mathrm{Ker}(\delta^1_{A|B})$. Indeed, let $\alpha_i\parall \alpha_j$ be an element in  $\mathfrak{I}$ and $\alpha_k\parall \alpha_l\in\mathrm{Ker}(\delta^1_{A|B})$, then
  \[
    [\alpha_i\parall \alpha_j,\alpha_k\parall \alpha_l]=\delta_{l,i} \alpha_k\parall \alpha_j-\delta_{j,k} \alpha_i\parall \alpha_l=\delta_{l,i} \alpha_k\parall \alpha_j\in \mathfrak{I},
  \]
  where $\delta_{i,j}$ denotes the Kronecker symbol. The last equality and the fact that $\alpha_k\parall \alpha_j\in \mathfrak{I}$ follows from observing that $m+1\leq k \leq n$ and $1\leq j \leq m$. If we assume that $\alpha_k\parall \alpha_l\in \mathfrak{I}$, then $\delta_{l,i}=0$ since in this case $m+1\leq i \leq n$ and $1\leq l \leq m$. Therefore, $\mathfrak{I}$ is an abelian ideal of $\mathrm{Ker}(\delta^1_{A|B})$. 

  If $n=1$, then $Q_A$ is either an $A_2$-quiver or a quiver with one loop. In both cases $m$ can be either $0$ or $1$. Recall that we have excluded the case when the characteristic of the field $k$ is $2$ and $Q_A$ is a quiver with one loop and $m=0$. Hence, if $m=0$, then $\HH^1(A|B)$ is an abelian Lie algebra of dimension $1$. If $m=1$, then $\HH^1(A|B)=0$.
\end{proof}

Note that in \cite[Propositions 4.1 and 4.4]{Sanchez08}, the author describes Lie algebra structure of $\HH^1(kQ)$, where $Q$ is the one loop and two loop quiver, respectively.

\begin{remark}\label{rem:char2loop}
  If the characteristic of the field $k$ is $2$ and $Q_A$ is a quiver with one loop $\alpha$ and $m=0$, then $\mathrm{Ker}(\delta^1_{A|B})$ is a solvable Lie algebra of dimension $2$ generated by $\alpha\parall  e$ and $\alpha\parall \alpha$. 
\end{remark}

\begin{lemma} 
  \label{lem:abelianrad2}
  Let $n$ be a positive integer. Let $A,B$ be the $n$ and the $(n-1)$-Kronecker quivers (or the $n$ and $(n-1)$-bouquets), respectively. Then the derived subalgebra of $\mathrm{Ker}(\delta^1_{A|B})$ is an abelian Lie algebra.
\end{lemma}

\begin{proof}
  If the characteristic of $k$ is $2$ and $Q_A$ is a quiver with one loop $\alpha$ and one vertex $e$, then by Remark \ref{lem:abelianrad2} the derived subalgebra of $\mathrm{Ker}(\delta^1_{A|B})$ is generated $\alpha\parall  e$, hence it is abelian. 
  Otherwise, it follows from Lemma \ref{lem:nkron} that $\mathrm{Ker}(\delta^1_{A|B})\cong \mathfrak{I}\rtimes k$, where $\mathfrak{I}\cong  k^{n-1}$. Note that the derived subalgebra of $\mathrm{Ker}(\delta^1_{A|B})$ is $\mathfrak{I}$. Indeed, $\mathfrak{I}$ is an abelian ideal and $[-,\alpha_n\parall \alpha_n]$ acts as identity on $\mathfrak{I}$.  
\end{proof}

\textbf{Notation.} Given a quiver $Q$, we denote by $\overline{Q}$ the quiver such that $\overline{Q}_0=Q_0$ and $\overline{Q}_1$ is the set of equivalence classes of parallel arrows. For each equivalence class $[\alpha]$ in $\overline{Q}$, we denote by $|\alpha|$ the cardinality of the equivalence class. 

Using a similar notation as in Lemma \ref{lem:nkron}, we introduce two subspaces of $\mathrm{Ker}(\delta^1_{A|B})$.  We denote by $Q_C$ the complement quiver of $Q_A$ by the arrows of $Q_B$. 
For each equivalence class $[\alpha]$ in $\overline{Q_C}$, we denote by $[\alpha'], [\alpha'']$ the corresponding equivalence classes in $\overline{Q_B}$ and in $\overline{Q_A}$, respectively. We order the arrows in $Q_A$ such that the first $|\alpha'|$ arrows are in $Q_B$ and the rest of the arrows are in $Q_C$. We denote by $ \mathfrak{I}_\alpha$ the Lie subalgebra of $\mathrm{Ker}(\delta^1_{A|B})$ having a $k$-basis given by 
\[
  \{\alpha_i\parall \alpha_j\ | \ \alpha_i\in [\alpha] \textrm { in } (\overline{Q_C})_1,\alpha_j\in [\alpha'] \textrm { in } (\overline{Q_B})_1, \ |\alpha'|+1\leq i\leq |\alpha''|,  \ 1\leq j\leq |\alpha'| \}.
\]
Note that $\mathfrak{I}_\alpha$ is trivial if the equivalence class $[\alpha']\in \overline{Q_B}$ is trivial. Similarly, for each equivalence class $[\alpha]$ in $\overline{Q_C}$, the Lie subalgebra of $\mathrm{Ker}(\delta^1_{A|B})$ having a $k$-basis given by  
\[
  \{\alpha_i\parall \alpha_j\ | \ \alpha_i ,\alpha_j \in [\alpha] \textrm { in } (\overline{Q_C})_1, \ |\alpha'|+1\leq i,j\leq |\alpha''|\}
\]
is isomorphic to $\mathfrak{gl}_{|\alpha|}(k)$.

\begin{proposition}
  \label{prop:kerrad}
  Let $A, B$ be radical square zero algebras such that $Q_A$ is not the quiver with one loop and $Q_B$ is a subquiver of $Q_A$. Then 
  \[
    \mathrm{Ker}(\delta^1_{A|B})\cong \prod_{[\alpha]\in (\overline{Q_C})_1} \mathfrak{I}_\alpha \rtimes \mathfrak{gl}_{|\alpha|}(k)
  \]
\end{proposition}

\begin{proof}
  Assume the characteristic of the field is different from $2$. A $k$-basis for $\mathrm{Ker}(\delta^1_{A|B})$ is given by $\{\alpha_i\parall \alpha_j\ | \ \alpha_i  \in Q_C,\ \alpha_j \in Q_A\}$. In order to study the Lie algebra structure of $\mathrm{Ker}(\delta^1_{A|B})$ it is enough to restrict to the classes of parallel arrows, see \cite{Str} after Remark 4.8. In other words, $\mathrm{Ker}(\delta^1_{A|B})$ decomposes, as a Lie algebra, to the direct product of Lie algebras given in Lemma \ref{lem:nkron}, where the product is indexed by equivalence classes in $(\overline{Q_C})_1$. 

  By Lemma \ref{lem:nkron}, the arguments of the previous paragraph can be applied also if the characteristic of the field $k$ is $2$ with the exception of $A$ having at least one vertex with exactly one loop. For a fixed vertex $e$ having a loop $\alpha$, we have $\alpha\parall e \notin \mathrm{Ker}(\delta^1_{A|B})$. Indeed, since $A$ is connected and $Q_A$ is not the quiver with one loop, there exists a relation $\alpha\beta$ or $\beta \alpha$ which forces $\alpha\parall e \notin \mathrm{Ker}(\delta^1_{A|B})$. If $\alpha\in Q_C$, then the corresponding Lie algebra for the equivalence class of $\alpha$ is generated only by $\alpha\parall \alpha$. If $\alpha\notin Q_C$, then the Lie algebra is trivial. Hence, the same decomposition also holds in this case.
\end{proof}

\textbf{Notation.} For $[\alpha] \in (\overline{Q_C})_1$, we denote by $I_{\alpha}=\sum _{a\in [\alpha]} a\parall a$ and $kI_{\alpha}$ the corresponding $k$-vector space.

\begin{proposition}
  \label{prop:rad}
  Let $A$ be radical square zero algebra and let $B \subseteq A$ be a radical square zero subalgebra over a field $k$ of characteristic zero such that $Q_B$ is a subquiver of $Q_A$. Then 
  \[
    \rad(\mathrm{Ker}(\delta^1_{A|B})) \cong \prod_{[\alpha]\in (\overline{Q_C})_1} \mathfrak{I}_\alpha \rtimes kI_{\alpha}.
  \]
\end{proposition}

\begin{proof}
  Let $[\alpha]$ be an equivalence class in $ (\overline{Q_C})_1$, we denote by $[\alpha'], [\alpha'']$ the corresponding equivalence classes in $(\overline{Q_B})_1, (\overline{Q_A})_1$, respectively. Set $n=|\alpha''|$ and $m=|\alpha'|$. Then, using the same notation as in Example \ref{Example:deleteRows}, we have that $\mathfrak{I}_\alpha\cong \mathcal{A}$. In addition, $ kI_{\alpha}$ is isomorphic to the space of scalar matrices. The statement follows from Example \ref{Example:deleteRows}. 
\end{proof}

\begin{proposition}
  Let $A$ be a radical square zero algebra and let $B \subseteq A$ be a radical square zero subalgebra over a field $k$ of characteristic zero such that $Q_B$ is a subquiver of $Q_A$. Then $\mathrm{Im}(\delta^0_{A|B})$ is an abelian Lie ideal of $\mathrm{R}(\mathrm{Ker}(\delta^1_{A|B}))$. 
\end{proposition}

\begin{proof}
  By Remark \ref{rem:relcomb} we have $\mathrm{Im}(\delta^0_{A|B})\subseteq \mathrm{Im}(\delta^0_{A})$ and by Lemma \ref{relcomb}, $\mathrm{Im}(\delta^0_{A|B})$ is a Lie  ideal of $\mathrm{Ker}(\delta^1_{A|B})$. If $Q_A$ is not an oriented cycle, then by \cite[Lemma 2.6]{Sanchez} we have that $\mathrm{Im}(\delta^0_{A})$ is an abelian ideal of $\mathrm{R}(\mathrm{Ker}(\delta^1_{A}))$.  By the same argument as in \cite[Lemma 2.6]{Sanchez}, $\mathrm{Im}(\delta^0_{A})$ is an abelian ideal of $\mathrm{R}(\mathrm{Ker}(\delta^1_{A}))$ also in the case when $Q_A$ is an oriented cycle. Hence $\mathrm{Im}(\delta^0_{A|B})$ is an abelian ideal of $\mathrm{Ker}(\delta^1_{A|B})$. In particular, $\mathrm{Im}(\delta^0_{A|B})$ is an abelian ideal of $\mathrm{R}(\mathrm{Ker}(\delta^1_{A|B}))$. 
\end{proof}

We recall a basic result of Lie algebras \cite[Lemma 2.7]{Sanchez}:

\begin{lemma}
  \label{lem:solv}
  Let $\mathfrak{g}$ be a Lie algebra and $\mathfrak{J}$ a solvable ideal of $\mathfrak{g}$. Then $\mathrm{R}(\mathfrak{g}/\mathfrak{J})=\mathrm{R}(\mathfrak{g})/\mathfrak{J}$.
\end{lemma}

We denote by
\[
  S= \{[\alpha]\in(\overline{Q_C})_1\  |\ |\alpha|>1\}
\]
and we call it the parallel complement.

\begin{proposition}\label{prop:semisimplerad2}
  Let $A$ be a radical square zero algebra and let $B \subseteq A$ be a radical square zero subalgebra over a field $k$ of characteristic zero such that $Q_B$ is a subquiver of $Q_A$. Then \[
    \HH^1(A|B)/ \rad( \HH^1(A|B))\cong \prod_{[\alpha]\in S} \mathfrak{sl}_{|\alpha|}(k) 
  \]
\end{proposition}

\begin{proof}
  By Lemma \ref{lem:solv} and the third isomorphism theorem we have that the semisimple part of $\HH^1(A|B)$ is isomorphic to $\mathrm{Ker}(\delta^1_{A|B})$ quotient by $\rad (\mathrm{Ker}(\delta^1_{A|B}))$. By Proposition \ref{prop:rad} and Proposition \ref{prop:kerrad} the statement follows. 
\end{proof}

The following proposition could also be proved using Theorem \ref{radsqrzero}, however, we use a more indirect method.

\begin{proposition}\label{prop:HH1rad2ab}
  Let $A, B$ be radical square zero algebras. Assume $Q_B$ is a subquiver of $Q_A$ such that $Q_C$ has no parallel arrows and at most one loop for each vertex. 
  Then the derived subalgebra of $\HH^1(A|B)$ is abelian.
\end{proposition}

\begin{proof}
  Let $A, B$ be radical square zero algebras such that $Q_A$ is not a quiver with one loop and $Q_B$ is a subquiver of $Q_A$.    It follows from Proposition \ref{prop:kerrad} that
  \[
    [\mathrm{Ker}(\delta^1_{A|B}), \mathrm{Ker}(\delta^1_{A|B})]=\prod_{[\alpha],[\beta]\in (\overline{Q_C})_1} [ \mathfrak{I}_\alpha \rtimes k, \mathfrak{I}_\beta \rtimes k]=\prod_{[\alpha]\in (\overline{Q_C})_1} \mathfrak{I}_{\alpha}
  \] 
  which is an abelian Lie algebra.  Note that the last equality follows from the proof of Lemma \ref{lem:abelianrad2}.
  If the characteristic of the field is $2$ and $Q_A$ is a quiver with one loop and one vertex then by Lemma \ref{lem:abelianrad2} the derived subalgebra of $\mathrm{Ker}(\delta^1_{A|B})$ is abelian.
  Let $\varphi: \mathrm{Ker}(\delta^1_{A|B})\twoheadrightarrow \HH^1(A|B)$ be the canonical projection. Since $\varphi$ is a surjective Lie algebra homomorphism,  the image of the derived subalgebra of  $\mathrm{Ker}(\delta^1_{A|B})$ is equal to the derived subalgebra of $\HH^1(A|B)$, see \cite[Exercise 4.1]{Erdmann}.
  Since the derived subalgebra of $\mathrm{Ker}(\delta^1_{A|B})$ is abelian it follows that the derived subalgebra of $\HH^1(A|B)$ is abelian.
\end{proof}

Let $[\alpha]$ be an equivalence class in $(\overline{Q_C})_1$. Then the corresponding equivalence class $[\alpha'']$ in $(\overline{Q_A})_1$  either contains no arrows in $Q_B$ or contains at least one arrow in $Q_B$. The set of equivalence classes in $(\overline{Q_C})_1$ of the first kind are denoted by $\overline{D}$, while the set of equivalence classes of the second kind are denoted by $\overline{E}$.

\begin{theorem}\label{radsqrzero}
  Let $A$ be a radical square zero algebra and let $B \subseteq A$ be a radical square zero subalgebra over a field $k$ of characteristic zero such that $Q_B$ is a subquiver of $Q_A$. Then
  \[
    \HH^1(A|B)\cong \prod_{[\alpha] \in S} \mathfrak{sl}_{|\alpha|}(k)\rtimes\left(\left(\prod_{[\alpha]\in \overline{E} }\mathfrak{I}_{\alpha}\rtimes kI_{\alpha} \right)\times \mathfrak{a} \right),
  \]
  where $\mathfrak{a}=(\prod_{[\alpha]\in \overline{D}} kI_{\alpha})/\mathrm{Im}
  (\delta^0_{A|B})$ is an abelian Lie subalgebra.
\end{theorem}

\begin{proof}
  By Proposition \ref{prop:semisimplerad2} and the Levi decomposition, see Theorem \ref{thm:Levi}, we just need to describe the radical of $\HH^1(A|B)$.  
  The radical of $\mathrm{Ker}(\delta^1_{A|B})$ decomposes as the direct product of $\mathfrak{g}_1:=\prod_{[\alpha]\in \overline{E} }\mathfrak{I}_{\alpha}\rtimes kI_{\alpha}$ and $\mathfrak{g}_2:= \prod_{[\alpha]\in \overline{D}} kI_{\alpha}$.
  If $[\alpha]$ is in $\overline{E}$ then there will be no corresponding relative inner derivation coming from the source or the target of $\alpha$. Hence, 
  \[
    \rad (\mathrm{Ker}(\delta^1_{A|B}))/\mathrm{Im}
    (\delta^0_{A|B})\cong \mathfrak{g}_1 \times \mathfrak{g}_2/\mathrm{Im}
    (\delta^0_{A|B}).
  \]
  Clearly, the Lie algebra $\mathfrak{a}=\mathfrak{g}_2/\mathrm{Im}
  (\delta^0_{A|B})$ is abelian. 
\end{proof}

\begin{remark}
 A similar decomposition of \cref{radsqrzero} has been obtained for the first Hochschild cohomology of toupie algebras \cite[Corollary 6.7]{ALS}.
\end{remark}

We provide now some examples. Recall that if $B$ is semisimple then $\HH^1(A|B)\cong \HH^1(A)$. The following example shows that the converse is not true in general.

\begin{example}
  Let $A$ be the radical square zero algebra 
  \[
    \begin{tikzcd}
      1 & 2 & 3
      \arrow["{\alpha_1}", shift left=2, from=1-1, to=1-2]
      \arrow["{\alpha_2}"', shift right=2, from=1-1, to=1-2]
      \arrow["\beta"', from=1-2, to=1-3]
    \end{tikzcd}
  \]
  and $B$ be
  \[
    \begin{tikzcd}
      1 & 2 & 3
      \arrow["\beta"', from=1-2, to=1-3]
    \end{tikzcd}
  \]
  Then $\HH^1(A|B)\cong \HH^1(A)$. Note that in this case $\overline{D}=\{[\alpha]\}$ and $\overline{E}=\varnothing$, where $[\alpha]=\{\alpha_1, \alpha_2\}$.
\end{example}

\begin{example}
  \label{Example2}
  Let $\A$ be the path algebra of the quiver 
  \begin{center}
    \begin{tikzpicture}[commutative diagrams/every diagram]
      \node at (0,0) {$1$};
      \node at (2,0) {$2$};
      \node at (4,0) {$3$};
      \node at (6,0) {$4$};
      
      \path
      (1.75,0.1) edge[->] node[anchor = south] {$\alpha_1$} (0.25,0.1)
      (1.75,-0.1) edge[->] node[anchor = north] {$\alpha_2$} (0.25,-0.1)
      (4.25,0.1) edge[->] node[anchor = south] {$\beta_1$} (5.75,0.1)
      (4.25,-0.1) edge[->] node[anchor = north] {$\beta_2$} (5.75,-0.1)
      (3.75,0.1) edge[->] node[anchor = south] {$\delta_1$} (2.25,0.1)
      (3.75,-0.1) edge[->] node[anchor = north] {$\delta_2$} (2.25,-0.1);
    \end{tikzpicture}
  \end{center}
  quotient by the ideal $I = (\alpha_i\delta_j)$ and let $B$ be the path algebra of the subquiver 
  \begin{center}
    \begin{tikzpicture}[commutative diagrams/every diagram]
      \node at (0,0) {$1$};
      \node at (2,0) {$2$};
      \node at (4,0) {$3$};
      \node at (6,0) {$4.$};
      
      \path
      (1.75,-0.1) edge[->] node[anchor = north] {$\alpha_2$} (0.25,-0.1)
      (4.25,0.1) edge[->] node[anchor = south] {$\beta_1$} (5.75,0.1)
      (4.25,-0.1) edge[->] node[anchor = north] {$\beta_2$} (5.75,-0.1);
    \end{tikzpicture}
  \end{center}

  Since $\mathrm{Im}(\delta^0_{A|B})$ is generated by $\delta_1\parall \delta_1+\delta_2\parall \delta_2$, then \(\HH^1(A|B) \cong \mathfrak{sl}_{|\delta|}(k) \times \mathfrak{I}_{|\alpha|}\rtimes \mathfrak{gl}_{|\alpha|}\cong \mathfrak{sl}_{2}(k) \times k\rtimes k\).
\end{example}

\section{Embedding and Lie structure for the first relative Hochschild cohomology}
\label{sec:embandsol}

In this section, we show that the first relative Hochschild cohomology embeds into the first Hochschild cohomology as a Lie algebra. We also give a sufficient condition for the solvability of the first relative Hochschild cohomology.

\begin{lemma}
  \label{lemma:emb}
  Let \(k\) be a field, $A$ a $k$-algebra and \(B \subseteq A\) be a \(k\)-subalgebra. Then we have an injective Lie algebra map
  \begin{equation*}
    \iota : \HH^1(A|B) \hookrightarrow \HH^1(A).
  \end{equation*}
\end{lemma}

\begin{proof}
  Note that we have a well-defined map 
  $$
  \iota : \HH^1(A|B) \rightarrow \HH^1(A)
  $$
  where $f \mapsto f$. Let $f,g \in \HH^1(A|B)$ and assume that $\iota(f) = \iota(g)$. This means that $f-g = \delta \in \Inn_k(A)$, but then this implies that $\delta \in \Inn_{B^e}(A)$ as $0 = f(B)-g(B) = \delta(B)$. Therefore, $\iota$ is an injective map, and as the Lie structure is given by the commutator, we see that $\iota$ is an injective Lie algebra map. 
\end{proof}

Using this we can now show the following

\begin{corollary}
  \label{relabrel}
  Let \(k\) be a field, $A$ a $k$-algebra and \(B \subseteq A\) be a \(k\)-subalgebra. Then the following holds
  \begin{enumerate}
  \item If \(\HH^1(A)\) is abelian then \(\HH^1(A|B)\) is abelian;
  \item If \(\HH^1(A)\) is nilpotent then \(\HH^1(A|B)\) is nilpotent;
  \item If \(\HH^1(A)\) is solvable then \(\HH^1(A|B)\) is solvable. 
  \end{enumerate}
\end{corollary}

Note that the converse does not hold in general, that is, $\HH^1(A)$ can be non-solvable but $\HH^1(A|B)$ can be solvable:

\begin{example}
  \label{Ex:notsolv}
  Assume $k$ is a field of characteristic different from $2$.  Let $\A = kQ/I$ where $Q$ is the quiver given by 
  \[\begin{tikzcd}
      1 & 2
      \arrow["{\alpha_{i=1,2}}", bend left, shift left=2, from=1-1, to=1-2]
      \arrow[bend left, from=1-1, to=1-2]
      \arrow["\beta", bend left, from=1-2, to=1-1]
    \end{tikzcd}\]
  and $I = (\alpha_i\beta)$. Let $B$ be the Kronecker quiver given by $\alpha_1$ and $\alpha_2$. Then $\HH^1(A)\cong \sl_2(k)\times k$ and $\HH^1(A|B)\cong k$.
\end{example}

We provide now another sufficient condition for the solvability of  $\HH^1(A|B)$ which also includes Example \ref{Ex:notsolv}. Recall that we denote by $Q_C$ the complement of the Gabriel quiver of $A$ by the arrows of the Gabriel quiver of $B$. Note that if $B$ is a subalgebra of $A$ such that $Q_B \subseteq Q_A$, then $(Q_A)_0=(Q_B)_0$, i.e. $A$ and $B$ have the same semisimple part.

\begin{theorem}
  \label{thm1}
  Let $k$ be an algebraically closed field of characteristic $0$. Let $A$ be a finite-dimensional basic algebra.  Let $B$ be a subalgebra of $A$ such that $Q_B \subseteq Q_A$.  If $Q_C$ has no parallel arrows and at most one loop at each vertex, then $\HH^1(A|B)$ is a solvable Lie algebra. 
\end{theorem}

\begin{proof} 
  In characteristic zero, every derivation preserves the Jacobson radical, see \cite[Theorem 4.2]{Hoc42}. Hence, we have a Lie algebra map $\varphi_A:\HH^1(A)\to \HH^1(A/J(A)^2)$. By Lemma \ref{lemma:emb} we have an embedding $\HH^1(A|B)\xhookrightarrow{} \HH^1(A)$. Hence, we have a Lie algebra homomorphism \[
    \varphi_{A|B}:\HH^1(A|B) \to \HH^1(A/J(A)^2|B/J(B)^2)
  \] 
  as $\varphi_A(\iota(f))(B/J(B)^2) = 0$ for any $f \in \HH^1(A|B)$. Since $\mathrm{Ker}(\varphi_A)$ is a nilpotent Lie ideal of $\HH^1(A)$, see \cite[Proposition 2.9.]{LinRyD}, we have that $\mathrm{Ker}(\varphi_{A|B})$ is a nilpotent Lie ideal of $\HH^1(A|B)$. By Proposition \ref{prop:HH1rad2ab} we have that the derived subalgebra of $\HH^1(A/J(A)^2|B/J(B)^2)$ is an abelian Lie algebra. Since the derived subalgebra of $\mathrm{Im}(\varphi_{A|B})$ is contained in the derived subalgebra of $\HH^1(A/J(A)^2|B/J(B)^2)$, which is an abelian Lie algebra, this implies that $\mathrm{Im}(\varphi_{A|B})$ is a solvable Lie algebra. The statement follows from $\HH^1(A)/ \mathrm{Ker}(\varphi_{A|B})$ and $\mathrm{Ker}(\varphi_{A|B})$ being solvable Lie algebras \cite[Lemma 4.4]{Erdmann}.  
\end{proof}

In positive characteristic we need to add a further assumption on $Q_A$:

\begin{theorem}
  \label{thm2}
  Let $k$ be an algebraically closed field of positive characteristic. Let $A$ be a finite-dimensional basic algebra such that the quiver of $A$ does not have any loops. Let $B$ be a subalgebra of $A$ such that $Q_B \subseteq Q_A$.  If $Q_C$ is a simple directed graph, then $\HH^1(A|B)$ is a strongly solvable Lie algebra.
\end{theorem}

Note that if $Q_A$ does not have any loops, then the condition of $Q_C$ having at most one loop and no parallel arrows is equivalent of requiring $Q_C$ to be a simple directed graph.
\begin{proof}
  Since the quiver of $A$ does not have any loops then by \cite[Lemma 2.6.]{LinRyD} every derivation preserves the Jacobson radical. By the same arguments as in the proof of Theorem \ref{thm1} the statement follows.
\end{proof}

\section{Dual extensions of directed monomial algebras}
\label{dualext}

Let $A, B$ be directed monomial algebras. In this section we compute the first Hochschild cohomology of a dual extension $\Lambda (B,A^{\op})$ in terms of $\HH^1(B)$ and $\HH^1(A|B)$.

\begin{definition}
  Let $A=Q_A/I_A$ and $B=Q_B/I_B$ be two finite-dimensional algebras such that the set of vertices of $Q_A$ and $Q_B$ coincide. Then the \textit{dual extension algebra} of $B$ and $A$ is $\Lambda(B,A^{\op})=kQ/I$,  where $Q$ and $I$ are defined as follows:
  \begin{itemize}
  \item $Q$ has the same set of vertices as $Q_A$ and $Q_B$.
  \item the set of arrows of $Q$ is the union of the set of arrows of $Q_A$ and of $Q_B$.
  \item if $I_A$ is generated by the relations $\rho_i$ and $I_B$ by the relations $\rho^{\prime}_j$, then $I$ is generated by $\{\rho_i, \rho^{\prime}_j,  \alpha\beta^{\prime}\}$, where $\alpha$ is an arrow of $Q_B$ and $\beta^{\prime}$  is an arrow of $Q_A$. We denote these sets of relations by $\rho_{B}, \rho_{A}$ and $\overline{\rho}_{\Lambda}$, respectively. 
  \end{itemize}
\end{definition}

\begin{example}\label{ex:runningExampleSection4}
  Let $B$ be the Kronecker quiver and let $A$ be the $A_2$-quiver. Then we get that $\Lambda = \Lambda(B,A^{\op})$ is the algebra in Example~\ref{Ex:notsolv}. 
\end{example}

\begin{definition}
  Let $Q$ be a finite and acyclic quiver $Q$ having vertices labelled by $\{1, \dots, n\}$. Without loss of generality, one can assume that for any arrow $\alpha: i\to j$ we have $i< j$. Let $I$ be an admissible
  ideal and set $B=KQ/I$. Then we say that $B$ is \textit{directed} if $Q$ is finite, acyclic, and numbered as above.
\end{definition}

Let $A$ and $B$ be directed algebras. Then $\Lambda(B,A^{\op})$ is a quasi-hereditary algebra having an exact Borel subalgebra $B$, see \cite[Example 1.6]{Xi1994}. Let $A$ and $B$ be directed monomial algebras. We set $\Lambda:=\Lambda(B,A^{\op})$ when it is clear from the context. 

Note that if $A$ and $B$ are monomial algebras, then $\Lambda$ is a monomial algebra. Therefore, in order to describe $\HH^1(\Lambda)$, $\HH^1(B)$, and $\HH^1(\Lambda|B)$, we use Proposition \ref{Stra2.6} and the notation therein. From now on $A$ and $B$ will be directed monomial algebras, with the same ordering on the vertices.

We first describe a basis for $k((Q_{\Lambda})_0\parall  \B_{\Lambda})$. It is clear that $\B_{\Lambda}=\B_B\cup\B_{A^{\op}}\cup \overline{\B}_{\Lambda}$, where $\overline{\B}_{\Lambda}$ is given by paths $\overline{p}$  where $\overline{p}=q_n\cdots q_1p_m\cdots p_1$ such that $q_i$ are arrows in $ \B_{A^{\op}}$ and $p_j$ are arrows in $ \B_B$.  Since, $(Q_{\Lambda})_0\parall \B_B = (Q_{\Lambda})_0\parall \B_{A^{\op}}= (Q_{\Lambda})_0\parall (Q_{\Lambda})_0$, we have that a basis for $k((Q_{\Lambda})_0\parall  \B_{\Lambda})$ is given by:
\[
  (Q_{\Lambda})_0\parall \B_{\Lambda}= (Q_{\Lambda})_0\parall (Q_{\Lambda})_0 \cup  (Q_{\Lambda})_0\parall \overline{\B}_{\Lambda}
\]
Let $e\in (Q_{\Lambda})_0$.  We denote by $\mathrm{Im}(\delta^0_{\Lambda})_{|_1}$ the image generated by $\delta_{\Lambda}^0(e\parall e)$. We describe now $k((Q_{\Lambda})_1\parall  \B_{\Lambda})$. Since $(Q_{\Lambda})_1=(Q_{B})_1 \cup (Q_{A^{\op}})_1 $, we have

\begin{equation*}
  \begin{split}
    (Q_{\Lambda})_1\parall \B_{\Lambda}= &(Q_B)_1\parall \B_B \cup (Q_{A^{\op}})_1\parall \B_{A^{\op}} \cup (Q_B)_1\parall \B_{A^{\op}} \cup (Q_{A^{\op}})_1\parall \B_{B}\\
    &\cup (Q_B)_1\parall \overline{\B}_{\Lambda} \cup (Q_{A^{\op}})_1\parall \overline{\B}_{\Lambda}.
  \end{split}
\end{equation*}

\begin{lemma}
  Let $A$ and $B$ be directed monomial algebras and let $\Lambda=\Lambda(B,A^{\op})$. If $\alpha\in (Q_{A^{\op}})_1$, then there does not exist a path $p$ in $\B_B$ such that $\alpha$ is parallel to $p$. Hence, $(Q_{A^{\op}})_1\parall  \B_B=\varnothing$. Similarly, $(Q_{B})_1\parall  \B_{A^{\op}}=\varnothing$.
\end{lemma}

\begin{proof}
  Assume $s(\alpha)=i=s(p)$ and $t(\alpha)=j=t(p)$ 
  where $p=p_n\cdots p_1$ such that $p_i\in (Q_B)_1$. Since $s(\alpha)=i> j=t(\alpha)$, then
  $s(p)>t(p)$. This is a contradiction since $p\in (Q_B)_1$.
\end{proof}

As a consequence of the previous lemma, we obtain the union of pairwise disjoint sets:
\begin{equation*}\label{basisdecomposition1}
  \begin{split}
    (Q_{\Lambda})_1\parall \B_{\Lambda}= (Q_B)_1\parall \B_B \cup   (Q_B)_1\parall \overline{\B}_{\Lambda} \cup
    (Q_{A^{\op}})_1\parall \B_{A^{\op}} \cup
    (Q_{A^{\op}})_1\parall \overline{\B}_{\Lambda}
  \end{split}
\end{equation*}
which gives
\begin{equation}\label{basisdecomposition2}
  \begin{split}
    k((Q_{\Lambda})_1\parall \B_{\Lambda})= k((Q_B)_1\parall \B_B) \oplus   k((Q_B)_1\parall \overline{\B}_{\Lambda}) \oplus
    k((Q_{A^{\op}})_1\parall \B_{A^{\op}}) \oplus
    k((Q_{A^{\op}})_1\parall \overline{\B}_{\Lambda}).
  \end{split}
\end{equation}
Note that
\[
  \mathrm{Ker}(\delta^1_{\Lambda|A^{\op}})=\mathrm{Ker}(\delta^1_{\Lambda})\cap(k((Q_B)_1\parall \B_B) \oplus   k((Q_B)_1\parall \overline{\B}_{\Lambda}) )
\]
and similarly
\[
  \mathrm{Ker}(\delta^1_{\Lambda|B})=\mathrm{Ker}(\delta^1_{\Lambda})\cap(k((Q_{A^{\op}})_1\parall \B_{A^{\op}}) \oplus   k((Q_{A^{\op}})_1\parall \overline{\B}_{\Lambda}).
\]

\begin{lemma}\label{lem:vectdec}
  Let $A$ and $B$ be directed monomial algebras and let $\Lambda=\Lambda(B,A^{\op})$. We have the following $k$-vector space decomposition 
  \[
    \mathrm{Ker}(\delta^1_{\Lambda})= \mathrm{Ker}(\delta^1_{\Lambda|A^{\op}})\oplus \mathrm{Ker}(\delta^1_{\Lambda|B}).
  \]
  In addition, we have the following decomposition 
  \begin{equation}
    \begin{split}
      \mathrm{\mathrm{Ker}(\delta^1_{\Lambda|A^{\op}})}&= 
      \mathrm{Ker}(\delta^1_B) \oplus (\mathrm{Ker}(\delta^1_{\Lambda})
      \cap   k((Q_B)_1\parall \overline{\B}_{\Lambda} )); \\
      \mathrm{Ker}(\delta^1_{\Lambda| B})&= \mathrm{Ker}(\delta^1_{A^{\op}}) \oplus (\mathrm{Ker}(\delta^1_{\Lambda})
      \cap   k((Q_{A^{\op}})_1\parall \overline{\B}_{\Lambda})), 
    \end{split}
  \end{equation}
  as $k$-vector spaces.
\end{lemma}

\begin{proof}
  We use the decomposition given in Equation \ref{basisdecomposition2}. Let $\alpha$ be an arrow of $Q_B$ and $p$ be a path in $B$, then
  \[
    \delta^1(\alpha\parall p) =\rho_B^{\alpha\parall p}+ \rho_{A^{\op}}^{\alpha\parall p}+ \overline{\rho}_{\Lambda}^{\ \alpha\parall p}=\rho_B^{\alpha\parall p}.
  \] 
  Let $\overline{p} \in \overline{\B}_{\Lambda}$. Then 
  \[
    \delta^1(\alpha\parall \overline{p})= \rho_B^{\alpha\parall \overline{p}}+ \rho_{A^{\op}}^{\alpha\parall \overline{p}}+ \overline{\rho}_{\Lambda}^{\ \alpha\parall \overline{p}}=\rho_B^{\alpha\parall \overline{p}}.
  \]
  Hence a basis for $\mathrm{Ker}(\delta^1_{\Lambda}) \cap k((Q_B)_1\parall \B_B)$ is given by linear combinations of elements of the form $\alpha\parall p $. Similarly, a basis of $\mathrm{Ker}(\delta^1_{\Lambda}) \cap k((Q_B)_1\parall \overline{\B}_{\Lambda} ))$ is given by linear combinations of elements of the form $\alpha\parall \overline{p}$.
  
  Let $\beta $ be an arrow of $Q_{A^{\op}}$ and $q$ be a path in $A^{\op}$. Then
  \[
    \delta^1(\beta\parall q) =\rho_{A^{\op}}^{\beta\parall q}
    \ \ \ \ \textrm{and} \ \ \ \ \delta^1(\beta\parall \overline{p})=\rho_{A^{\op}}^{\beta\parall \overline{p}}.
  \] 
  Therefore a basis of $\mathrm{Ker}(\delta^1_{\Lambda})\cap k((Q_{A^{\op}})_1\parall \B_{A^{\op}})$ is given by linear combinations of elements $\beta\parall q$. Similarly, a basis for $\mathrm{Ker}(\delta^1_{\Lambda})\cap  k((Q_{A^{\op}})_1\parall \overline{\B}_{\Lambda})$ is given by linear combinations of elements of the form $\beta\parall \overline{p}$.
\end{proof}

Denote 
$$\mathcal{J}^{\prime}:=(\mathrm{Ker}(\delta^1_{\Lambda})
\cap   k(Q_B)_1\parall \overline{\B}_{\Lambda} ).$$

\begin{lemma}
  Let $A$ and  $B$ be directed monomial algebras and let $\Lambda=\Lambda(B,A^{\op})$. We have the following decomposition as $k$-vector spaces
  \[
    \mathrm{Ker}(\delta^1_{\Lambda|A^{\op}})= \mathrm{Ker}(\delta^1_B) \oplus \mathcal{J}^{\prime}
  \]
  and $\mathcal{J}^{\prime}$ is a Lie ideal of $\mathrm{Ker}(\delta^1_{\Lambda|A^{\op}})$.
\end{lemma}

\begin{proof}
  The $k$-vector space decomposition follows from Lemma \ref{lem:vectdec}. We only need to show that $\mathcal{J}^{\prime}$ is a Lie ideal of $\mathrm{Ker}(\delta^1_{\Lambda|A^{\op}})$. Let $\alpha,\beta\in (Q_B)_1$, $p\in \B_B$ and $q\in \overline{\B}_{\Lambda}$. Clearly, we have $[\alpha\parall p, \beta\parall  \overline{q}] \in (Q_B)_1\parall \overline{\B}_{\Lambda}$. The statement follows. 
\end{proof}

The $k$-vector space $$\mathcal{J}:=\mathcal{J}^{\prime}\oplus\mathrm{Ker}(\delta^1_{\Lambda| B})$$ is a Lie ideal of $\mathrm{Ker}(\delta^1_{\Lambda})$. Indeed, let $\alpha\in (Q_B)_1$, $p\in \B_B$,  $\gamma\in (Q_{A^{\op}})_1$ and $r\in \overline{\B}_{\Lambda}$. Then $[\alpha\parall p, \gamma\parall r]\in (Q_{A^{\op}})_1\parall \overline{\B}_{\Lambda}$. We denote $$\mathcal{I}:=\mathrm{Im}(\delta^0_{\Lambda})_{|_{\geq 2}}.$$

\begin{remark} \label{rmk:iso}
  Since $\mathrm{HH}^1(\Lambda)$ is a graded Lie algebra, there is an isomorphism of Lie algebras $\mathrm{Im}(\delta^0_\Lambda)/\mathcal{I}\cong \mathrm{Im}(\delta^0_{\Lambda})_{|_1}$. Since $B$ is directed, then $\mathrm{Im}(\delta^0_B)_{|_1}=\mathrm{Im}(\delta^0_B)$. In addition, it follows from \cite[Lemma 3.5]{LRW} that the dimension of $\mathrm{Im}(\delta^0_{\Lambda})_{|_1}$ is equal to the dimension of $\mathrm{Im}(\delta^0_B)_{|_1}$. Since $\mathrm{Im}(\delta^0_{\Lambda})_{|_1}$ and $ \mathrm{Im}(\delta^0_B)$ are abelian Lie algebras of the same dimension, there is an isomorphism of Lie algebras  $\mathrm{Im}(\delta^0_{\Lambda})_{|_1}\cong \mathrm{Im}(\delta^0_B)$. We call the composition of these two isomorphisms $$\eta_1\colon \mathrm{Im}(\delta^0_{\Lambda})/\mathcal{I}\to \mathrm{Im}(\delta^0_B).$$ Similarly, there is an isomorphism of Lie algebras $$\eta_2\colon\mathrm{Ker}(\delta^1_\Lambda)/\mathcal{J}\to  \mathrm{Ker}(\delta^1_B).$$
\end{remark}

\begin{theorem}\label{Dualmain}
  Let $A$ and  $B$ be directed monomial algebras and let $\Lambda=\Lambda(B,A^{\op})$. We have an exact sequence of Lie algebras
  \begin{center}
    \begin{tikzcd}
      0 \arrow[r] &  \mathcal{J}/\mathcal{I} \arrow[r] & \HH^1(\Lambda) \arrow[r] & \HH^1(B)\arrow[r] & 0.
    \end{tikzcd}
  \end{center}
\end{theorem}

\begin{proof}
  We have the following commutative diagram
  \begin{equation}\label{diag:com}
    \begin{tikzcd}
      0 \arrow[r]& \mathcal{I} \arrow[r, "\iota_1"] \arrow[d, hook]
      & \mathrm{Im}(\delta^0_{\Lambda}) \arrow[d, hook] \arrow[r,"\eta_1"] & \mathrm{Im}(\delta^0_B) \arrow[d, hook] \arrow[r] & 0\\
      0 \arrow[r,] & \mathcal{J} \arrow[r, "\iota_2"]\arrow[d, twoheadrightarrow]
      &  \mathrm{Ker}(\delta^1_{\Lambda})  \arrow[r, "\eta_2"]\arrow[d, twoheadrightarrow] & \mathrm{Ker}(\delta^1_{B})  \arrow[r]\arrow[d, twoheadrightarrow] & 0\\
      0 \ar{r} &  \mathcal{J}/\mathcal{I} \arrow[r] & \HH^1(\Lambda) \arrow[r] & \HH^1(B) \ar{r} & 0 
    \end{tikzcd}
  \end{equation}
  where $\iota_1,\iota_2$ are the canonical inclusions. By Remark \ref{rmk:iso} the first two rows are exact sequences of Lie algebras. Applying the Nine lemma to Diagram \ref{diag:com} we have that the last row is an exact sequence of vector spaces. Since the maps are induced by Lie algebra maps, this is an exact sequence of Lie algebras.
\end{proof}

In degree one we have the following decomposition of $\HH^1(\Lambda)$:

\begin{theorem}\label{thm:deg-one}
  Let $A$ and  $B$ be directed monomial algebras and let $\Lambda=\Lambda(B,A^{\op})$. Then we have the following isomorphism as Lie algebras
  \[
    \HH^1(\Lambda)_{|_1}\cong \HH^1(B)_{|_1}\oplus \HH^1(\Lambda|B)_{|_1}.
  \]
\end{theorem}

\begin{proof}
  By Theorem \ref{Dualmain} we have an exact sequence of Lie algebras
  \begin{center}
    \begin{tikzcd}
      0 \arrow[r] &  \mathrm{Ker}(\delta^1_{\Lambda|B})_{|_1}\arrow[r] & \HH^1(\Lambda)_{|_1}   \arrow[r] & \HH^1(B)_{|_1}\arrow[r] & 0.
    \end{tikzcd}
  \end{center}
  By Lemma \ref{lem:vectdec} we have the decomposition of vector spaces: $\mathrm{Ker}(\delta^1_{\Lambda})_{|_1}= \mathrm{Ker}(\delta^1_{A^{\op}})_{|_1}\oplus \mathrm{Ker}(\delta^1_{B})_{|_1}$ which becomes a direct product of Lie algebras since there are no pairs of parallel arrows $\alpha \parall \beta$ such that $\alpha\in Q_A$ and $\beta\in Q_B$, or vice versa.  By Remark \ref{rmk:iso} we have an isomorphism of Lie algebras $\mathrm{Im}(\delta_{\Lambda}^0)_{|_1}\cong \mathrm{Im}(\delta_{B}^0)_{|_1}$. Since $\HH^1(\Lambda|B)_{|_1}=\mathrm{Ker}(\delta^1_{\Lambda|B})_{|_1}=\mathrm{Ker}(\delta^1_{A^{\op}})_{|_1}$, the statement follows.
\end{proof}

\begin{corollary}
  \label{bipartite}
  Assume that $A, B$ are radical square zero basic algebras such that $Q_A$ and $Q_B$ do not have loops. Let $\Lambda=\Lambda(B, A^{\op})$. Then we have 
  \[
    \HH^1(\Lambda)\cong \HH^1(B)\oplus \HH^1(\Lambda|B).
  \]
  In addition, $\HH^1(B)$ is isomorphic to a direct product of $\mathfrak{sl}_{n_{i,j}}$ where $n_{i,j}$ denotes the number of parallel arrows in $Q_B$ between the vertices $e_i$ and $e_j$. Similarly, $\HH^1(\Lambda|B)$ is isomorphic to a direct product of $ \mathfrak{gl}_{m_{i,j}}$, where $m_{i,j}$ denotes  the number of parallel arrows in $Q_A$ between two vertices $e_i$ and $e_j$.
\end{corollary}

\begin{proof}
  Consider the decomposition:
  \begin{equation*}
    \begin{split}
      k(Q_{\Lambda})_1\parall \B_{\Lambda}= k(Q_B)_1\parall \B_B \oplus   k(Q_B)_1\parall \overline{\B}_{\Lambda} \oplus
      k(Q_{A^{\op}})_1\parall \B_{A^{\op}} \oplus
      k(Q_{A^{\op}})_1\parall \overline{\B}_{\Lambda}
    \end{split}
  \end{equation*}
  We show that both $k(Q_B)_1\parall \overline{\B}_{\Lambda}$ and $k(Q_{A^{\op}})_1\parall \overline{\B}_{\Lambda}$ are zero. Indeed, since $A, B$ are radical square zero algebras, an element in $\overline{\B}_{\Lambda}$ is of the form $q_1p_1$ where $p_1,q_1$ are arrows in $Q_B, Q_{A^{\op}}$, respectively. Since there are no loops in $Q_A, Q_B$, then there are no arrows $\alpha$ in $Q_{A^{\op}}$ or in $Q_B$ parallel to $q_1p_1$. 

  It follows from Lemma \ref{lem:vectdec} that $\mathrm{Ker}(\delta^1_{\Lambda})= \mathrm{Ker}(\delta^1_{B})\oplus \mathrm{Ker}(\delta^1_{A^{\op}})$. Since $A,B$ are radical square zero algebras and since $Q_A$ and $Q_B$ do not have any loops,  then $\mathrm{Ker}(\delta^1_{\Lambda})\cong \mathrm{Ker}(\delta^1_{\Lambda})_{|_{1}}$.
  Consequently, $\HH^1(\Lambda) \cong \HH^1(\Lambda)_{|_{1}}$. Similarly, $\HH^1(B)\cong \HH^1(B)_{|_{1}}$ and $\HH^1(\Lambda|B)\cong \HH^1(\Lambda|B)_{|_{1}}= \mathrm{Ker}(\delta^1_{A^{\op}})$. The first part of the statement then follows from Theorem \ref{thm:deg-one}. The second part of the statement follows from \cite[Theorem 2.9]{Sanchez}.
\end{proof}

Recall that a finite quiver $Q$ is \textit {bipartite} if every vertex is a source or a sink. Corollary \ref{bipartite} holds in particular if $Q_A$ and $Q_B$ are bipartite. 

\begin{example}
  Let $A$, $B$ and $\Lambda$ be the algebras in Example~\ref{ex:runningExampleSection4}. We have that 
  \begin{equation*}
    \HH^1(B) \cong \mathfrak{sl}_2
  \end{equation*}
  and 
  \begin{equation*}
    \HH^1(\Lambda | B) \cong k
  \end{equation*}
  Therefore we get by Corollary~\ref{bipartite} that 
  \begin{equation*}
    \HH^1(\Lambda) \cong \mathfrak{sl}_2 \oplus k.
  \end{equation*}
\end{example}

\begin{example}\label{ex:string}
  Let $Q$ be a quiver such that the underlying graph is a tree. Then $\Lambda(kQ,kQ^{\op})$ has been considered for example in \cite[\S 5.1]{Conde}. In \cite{Z} the Ringel dual has been considered.  A particular  case of the previous example is when the quiver of $B$ is of type $A_n$, that is, 
  \[\begin{tikzcd}
      1 & 2 & \dots & n
      \arrow["{\alpha_1}", from=1-1, to=1-2]
      \arrow["{\alpha_2}", from=1-2, to=1-3]
      \arrow["{\alpha_{n-1}}", from=1-3, to=1-4]
    \end{tikzcd}\]
  Then the Hochschild cohomology of $\Lambda$ has been studied in \cite[Example 4]{dPX} and \cite[\S 2.4]{XZ}. In particular, we have
  $$
  \dim \HH^i(\Lambda)=\begin{cases}
    n & i=0\\
    \frac{1}{2}(n-1)n , & i=1,2\\
    0, & i>2
  \end{cases}
  $$
  A computation shows that 
  $$
  \dim \HH^i(\Lambda|B)=\begin{cases}
    n & i=0\\
    \frac{1}{2}(n-1)n , & i=1\\
    0, & i\geq 2
  \end{cases}
  $$
  Since $Q$ is a tree, then $\HH^1(kQ)=0$. It is easy to show that $\HH^1(\Lambda)\cong\HH^1(\Lambda|B)$ as Lie algebras. Indeed, by Theorem \ref{Dualmain} it is enough to show that $\mathcal{J}/\mathcal{I}\cong \HH^1(\Lambda| B)$. First note that in this case $\mathrm{Ker}(\delta^1_{\Lambda|B})=\HH^1(\Lambda|B)$. For each vertex $e_i$, consider a cycle $c_i$ in $\mathcal{B}_\Lambda$ starting and ending at $e_i$. Note that $c_i\in \overline{\mathcal{B}}_{\Lambda}$, that is, $c_i$ is of the form $q_h\cdots q_1p_k\cdots p_1$ such that $q_i$ are arrows in $ \B_{B^{\op}}$ and $p_j$ are arrows in $ \B_B$. If we denote for each arrow $\alpha_i$ the corresponding arrow of $\mathcal{B}_{B^{\op}}$ by $\alpha^*_i$, then $\delta^{0}(e_i\parall c_i)= \alpha_i\parall c_i\alpha_i -\alpha^*_i\parall \alpha^*_ic_i$. We observe that $\alpha_i\parall c_i\alpha_i\in \mathcal{J}'$ and $\alpha^*_i\parall\alpha^*_ic_i\in \mathrm{Ker}(\delta^1_{\Lambda|B})$. 

  Note that each $\alpha_i\parall\overline{p'}\in \mathcal{J}'$ is of the form $\alpha_i\parall c'_i\alpha_i$ for some cycle $c'_i$ in $\mathcal{B}_\Lambda$. Indeed, since $s(\alpha_i)=s(\overline{p'})=e_i$, then $\overline{p'}=q_m\dots q_1p_l\dots p_3 \alpha_{i+1}\alpha_i$ for some arrows $p_i$ in $\mathcal{B}_B$ and $q_j$ in $\mathcal{B}_{B^{\op}}$. Note that $\alpha_{i+1}$ must be an arrow of $\overline{p'}$, otherwise $t(\overline{p'})\neq e_{i+1}$. Similarly, since $t(\alpha_i)=t(\overline{p'})=e_{i+1}$, then $\overline{p'}=\alpha^*_{i+1}q_{m-1}\dots q_1p_l\dots p_3 \alpha_{i+1}\alpha_i$. Hence, $c'_i$ is the cycle $\alpha^*_{i+1} \dots \alpha_{i+1}$ and
  $\overline{p'}=c'_i\alpha_i$. 

  Therefore, there is a bijection between a basis of $\mathcal{J}'$ and a basis of $\mathcal{I}=\mathrm{Im}(\delta^0_{\Lambda})_{|_{\geq 2}}$. Consequently, $(\mathcal{J}'\oplus \mathrm{Ker}(\delta^1_{\Lambda|B}))/ \mathcal{I} \cong \HH^1(\Lambda|B)$.
  It follows from \cite[Corollary 4.12]{Str} that $\HH^1(\Lambda)$ is a solvable Lie algebra.
\end{example}

Note that in general $\HH^1(\Lambda| B)$ is not isomorphic to $\mathcal{J}/\mathcal{I}$ as the following example shows. 

\begin{example}\label{ex:notcong}
  Let $Q$ be the quiver
  \[\begin{tikzcd}
      & 2 \\
      1 & 3 & 4
      \arrow["{\alpha_2}"', from=1-2, to=2-2]
      \arrow["{\alpha_1}", from=2-1, to=2-2]
      \arrow["{\alpha_3}", from=2-2, to=2-3]
    \end{tikzcd}\]    
  and let $B=kQ$. Let $\Lambda=\Lambda(B, B^{\op})$. We will show that $\mathcal{J}/\mathcal{I}$ is 
  not isomorphic to $\HH^1(\Lambda|B)$. Denote the arrows of $\mathcal{B}_{B^{\op}}$ by $\alpha^*_i$. A basis of $\mathcal{I}=\mathrm{Im}
  (\delta^0_{\Lambda})_{|_{\geq 2}}$ is given by $\{\alpha_1\parall\alpha^*_3\alpha_3\alpha_1+\alpha_2\parall\alpha^*_3\alpha_3\alpha_2+
  \alpha_2^*\parall\alpha^*_2\alpha^*_3\alpha_3+ 
  \alpha^*_1\parall\alpha^*_1\alpha^*_3\alpha_3\}$. A basis of 
  $\mathcal{J}'$ is given by 
  $\{\alpha_1\parall \alpha_3^*\alpha_3\alpha_1, 
  \alpha_2\parall\alpha_3^*\alpha_3\alpha_2\}$. A basis of $\mathrm{Ker}
  (\delta^1_{\Lambda|B})$ is given by
  $$\{\alpha^*_1\parall\alpha^*_1,\alpha^*_2\parall\alpha^*_2, \alpha^*_3\parall\alpha^*_3, 
  \alpha^*_1\parall\alpha^*_1\alpha^*_3\alpha_3, 
  \alpha^*_2\parall\alpha^*_2\alpha^*_3\alpha_3\}.$$
  Since the dimension of $\mathcal{J}/\mathcal{I}$ is 6, while the dimension of $\HH^1(\Lambda|B)=\mathrm{Ker}
  (\delta^1_{\Lambda|B})$ is 5, the statement follows. 
\end{example}

\printbibliography 
\end{document}